\documentclass[12pt,a4paper]{article}
\usepackage[utf8]{inputenc}
\usepackage[T2A]{fontenc}
\usepackage[english]{babel}
\usepackage{amsmath,amssymb,amsfonts,amsthm}
\usepackage{geometry}
\usepackage{hyperref}
\usepackage{mathrsfs}
\hypersetup{
    colorlinks=true,
    linkcolor=blue,
    filecolor=magenta,      
    urlcolor=cyan,
}

\theoremstyle{plain}
\newtheorem{theorem}{Theorem}[section]
\newtheorem{lemma}[theorem]{Lemma}
\newtheorem{proposition}[theorem]{Proposition}

\theoremstyle{definition}
\newtheorem{definition}[theorem]{Definition}
\newtheorem{remark}[theorem]{Remark}

\newcommand{\HH}{\mathfrak{H}}
\newcommand{\KK}{\mathfrak{K}}
\newcommand{\RR}{\mathbb{R}}
\newcommand{\CC}{\mathbb{C}}
\newcommand{\norm}[1]{\left\lVert #1 \right\rVert}

\newcommand{\specp}{\sigma_{\mathrm{p}}}

\newcommand{\pv}{\text{p.v.}}

\title{\textbf{On singular points in the essential spectrum}}
\author{Alexander Plakhotnikov\footnote{Student. St. Petersburg State University. Faculty of Mathematics and Mechanics. \texttt{e-mail: st132770@student.spbu.ru}}}
\date{\today}

\begin{document}

\maketitle

\begin{abstract}
The paper investigates the existence of a limit in the operator norm for a family of operators $T_z(H)= F(H-z)^{-1}F^*$ for $z$ tending to the real axis. The conditions for the $H$ operator and the rigging operator $F$ are established, under which the limit exists. Special attention is paid to the separation of cases when the limit point belongs and does not belong to the point spectrum $H$.
\end{abstract}

\bigskip

{\bf 1. Introduction} \label{sec:introduction}
Consider a family of operators of the form $T_z(H) := F(H-z)^{-1}F^*$, where $H$ is a self—adjoint operator in a separable Hilbert space $\HH$, and $F:\HH\to\KK$ is the so-called rigging operator. This construction, often referred to as a sandwiched resolvent, finds applications in scattering theory (where $T_z$ is associated with the core of the scattering matrix), in the study of the spectral properties of perturbed operators (for example, in the Wigner-Weiss formalism \cite{Grubb1996}) and in the study of the local spectral properties of the operator $H$. From the boundary values of the resolvent on the real axis, i.e. from the limits of $T_{\lambda+iy}\to T_{\lambda+i0}$ at $y\downarrow 0$, information about the spectral density of the operator $H$ at the point $\lambda\in\RR$ can be extracted.

For the existence of the limit $T_{\lambda+i0}$ in the operator norm of the Hilbert space $\KK$, certain conditions are necessary both for the local structure of the spectrum of the operator $H$ near the point $\lambda$ and for the properties of the "equipping" operator $F$. In particular, in the classical theory of scattering by David Yafaev ~\cite{Yafaev1992}, it is shown that for Schrodinger-type operators with decreasing potentials, the existence of a limit $\lim_{y\downarrow 0}(H-z)^{-1}$ in a special locally weighted norm is related to the smoothness of the spectral measure.

This paper is devoted to the study of the conditions for the existence of the limit of the sandwiched resolvent $T_{\lambda+iy}$ in the operator norm of the space $\KK$ at $y\downarrow 0$. We consider the self-adjoint operator $H$ in $\HH$ and the bounded operator $F:\HH\to\KK$, defined in the spectral representation of $H$ through the multiplication operator by the weight function $w$ with a compact carrier. This requirement guarantees, in particular, the compactness of the operator $F(1+|H|)^{-s}$.

In this paper, we analyze cases where the limit point is $\lambda\in\RR$:

\begin{enumerate}
    \item does not belong to the point spectrum $\specp(H)$;
    \item is an eigenvalue of $\lambda\in\specp(H)$ of finite multiplicity.
\end{enumerate}

In Case I, we show that the existence of the limit $T_{\lambda+i0}$ in the operator norm $\KK$ is guaranteed if the spectral density $\rho$ of the operator $H$ and the weight function $w$ of the operator $F$ satisfy the local Holder condition at the point $\lambda$. In Case II, we prove that the presence of an eigenvalue of $\lambda$ with a nontrivial rigging operator (i.e., $F P_\lambda\neq 0$) leads to the divergence of the norm of $\|T_{\lambda+iy}\|$ as $O(1/y)$ for $y \downarrow 0$. In addition, it can be shown that the regularized part of the operator, which excludes the contribution of its own subspace, still has a limit.

The structure of this work is organized as follows. Section 2 provides preliminary definitions and establishes basic assumptions about the operators $H$ and $F$. Section 3 is devoted to proving the compactness property of the rigging operator. Section 4 analyzes the limit of $\lim_{y\downarrow 0} T_{\lambda+iy}$. The reader can find more detailed theoretical information in ~\cite{NA1,NA2,NA3,NA4,NA5,ReedSimon1}.

\bigskip

{\bf 2. Preliminary definitions.} Let $H$ be a self—adjoint operator in a separable Hilbert space $\HH$. We are interested in a family of operators depending on the complex parameter $z\in\CC\setminus\RR$:
\begin{equation}\label{eq:Tz}
T_z(H) := F(H-z)^{-1}F^*, \quad z=\lambda+iy, \ y \neq 0,
\end{equation}
where $F:\HH\to\KK$ is a restricted rigging operator. Our goal is to establish the conditions for the existence of the limit of this expression in the operator norm at $y\downarrow 0$ for a fixed $\lambda\in\RR$.

\begin{definition}
Let $\KK$ be a fixed separable Hilbert space. For concreteness, we can put $\KK := L^2(\RR,d\nu)$, where $\nu$ is some $\sigma$-finite Borel measure on $\RR$. 
\end{definition}

According to the spectral theorem, for $H$ there exists a unitary operator (spectral transformation) $\mathfrak{U}_H: \HH\to\bigoplus_{k=1}^{N}L^2(\RR,d\mu_k)$, where $N\in\mathbb{N} \cup\{\infty\}$, and $\{\mu_k\}$ is a set of measures that determine the spectral properties of $H$. For simplicity, we will consider the case of $N=1$, that is, $\mathfrak{U}_H:\HH\to L^2(\RR,d\mu_H)$. All arguments are generalized to the case of an arbitrary multiplicity of the spectrum.

\begin{definition}
Let $w:\RR\to (0,\infty)$ be a weight function satisfying the condition $w\in C_0(\RR)$ . We define the rigging operator $F:\HH\to\KK$ as a composition:
\begin{equation}\label{eq:F}
F := J \circ M_w \circ \mathfrak{U}_H,
\end{equation}
where:
\begin{enumerate}
    \item $\mathfrak{U}_H: \HH \to L^2(\RR, d\mu_H)$ is the spectral transformation for $H$.
    \item$M_w: L^2(\RR, d\mu_H) \to L^2(\RR, d\mu_H)$ is the multiplication operator for the function $w$,\\$(M_w f)(x) = w(x)f(x)$.
    \item$J: L^2(\RR,d\mu_H) \to\KK$ is a restricted operator, for example, an isometric embedding. 
\end{enumerate}
In the spectral representation, the operator $F$ acts on the vector $u\in\HH$ with the image $\widehat{u}=\mathfrak{U}_H u$ as $Fu\mapsto J(w\widehat{u})$.
\end{definition}

\bigskip

 {\bf 3. Compactness.} For further analysis, we need the compactness of the operator connecting $F$ with decreasing resolution at infinity.

\begin{lemma}
\label{lem:compactness}
Let the weight function be $w\in C_0(\RR)$. Then for any $s > 1/2$, the operator $F(1+|H|)^{-s}:\HH\to\KK$ is compact.
\end{lemma}

\begin{proof}
Consider the operator $K := F(1+|H|)^{-s}$. Using the functional calculus for $H$, we can represent $K$ in a spectral representation. The operator $(1+|H|)^{-s}$ is unitarily equivalent to the multiplication operator by the function $(1+|x|)^{-s}$ in the space $L^2(\RR,d\mu_H)$. Thus,
\[
K =J \circ M_w \circ \mathfrak{U}_H \circ\mathfrak{U}_H^{-1} \circ M_{(1+|x|)^{-s}} \circ\mathfrak{U}_H =J\circ M_{w(x)(1+|x|)^{-s}} \circ\mathfrak{U}_H.
\]
We denote $g(x) = w(x)(1+|x|)^{-s}$. Since $w\in C_0(\RR)$, the function $g(x)$ is also continuous and $g(x)\to 0$ for $|x|\to \infty$, that is, $g\in C_0(\RR)$.

The multiplication operator $M_g$ by the function $g\in C_0(\RR)$ in $L^2(\RR,d\mu_H)$ is compact. This is a classic result: $g$ can be approximated in a uniform norm by functions with compact support $g_n\in C_0(\RR)$. The multiplication operators $M_{g_n}$ have finite rank (if the measure $\mu_H$ is finite on the carrier $g_n$) or are compact in the general case. Since $\|M_g - M_{g_n}\| = \|g - g_n\|_\infty\to 0$, the operator $M_g$ is the limit of compact operators and, therefore, is compact itself.

Since $K$ is a composition of the compact operator $M_g$ with bounded operators $J$ and $\mathfrak{U}_H$, it is also compact.
\end{proof}

\bigskip

{\bf 4. Limit analysis $\lim_{y\downarrow 0} T_{\lambda+iy}$.} We will use a standard technique, dividing the spectrum into "near" and "far" zones relative to the point $\lambda$. For $\varepsilon > 0$, we define:
\[
I_{\mathrm{near}} := (\lambda-\varepsilon, \lambda+\varepsilon), \qquad I_{\mathrm{far}} := \RR \setminus I_{\mathrm{near}}.
\]
Let's denote the corresponding spectral projectors by $E_{\mathrm{near}} := E_H(I_{\mathrm{near}})$ and $E_{\mathrm{far}} := E_H(I_{\mathrm{far}})$.

The resolvent can be decomposed into two parts:
\[
(H-z)^{-1} = (H-z)^{-1}E_{\mathrm{far}} + (H-z)^{-1}E_{\mathrm{near}}.
\]
On the far side, $\operatorname{Im}E_{\mathrm{far}}$, the resolvent is uniformly bounded by $y$ for $y\in(0,1)$:
\[
\norm{(H-z)^{-1}E_{\mathrm{far}}} = \sup_{x \in I_{\mathrm{far}}} \frac{1}{|x-\lambda-iy|} \le \frac{1}{\varepsilon}.
\]
Therefore, the operator $F(H-\lambda-iy)^{-1}E_{\mathrm{far}}F^*$ converges in the operator norm to $F(H-\lambda)^{-1}E_{\mathrm{far}}F^*$ for $y \downarrow 0$ by the dominated convergence theorem for operator integrals. The main difficulty lies in analyzing the contribution from $I_{\mathrm{near}}$.

\bigskip

{\bf 4.1. Case I: $\lambda\notin\specp(H)$ (no eigenvalue).} Assume that $\lambda$ is not an eigenvalue of $H$. For the existence of a limit in the operator norm, we need a condition for the regularity of the continuous spectrum $H$ in the vicinity of $\lambda$.

\begin{definition}
We will say that the spectral measure $\mu_H$ satisfies the \emph{H\"older condition} at a point $\lambda$ if it is absolutely continuous in some neighborhood of $\lambda$, $d\mu_H(x) = \rho(x)dx$, and its density $\rho(x)$ is locally H\"olderian at the point $\lambda$ with exponent $\alpha\in (0, 1]$. That is, $|\rho(x) - \rho(\lambda)| \le C|x-\lambda|^\alpha$.
\end{definition}

\begin{theorem}[see ~\cite{Yafaev1992}]
Let $\lambda\notin\specp(H)$. Assume that the spectral measure $\mu_H$ satisfies the H\"older condition at the point $\lambda$, and the weight function $w\in C_0(\RR)$ is also locally H\"olderic in $\lambda$. Then the limit
is \[
T_{\lambda+i0} := \lim_{y\downarrow0} F(H-\lambda-iy)^{-1}F^*
\]
exists in the operator norm.
\end{theorem}

\begin{proof}[Proof scheme]
As noted above, the contribution from $I_{\mathrm{far}}$ converges. Consider the contribution from $I_{\mathrm{near}}$. In the spectral representation, it corresponds to the operator $J\circ K_y\circ J^*$, where $K_y$ is an integral operator in $L^2(I_{\mathrm{near}}, d\mu_H)$ with the kernel
\[
K_y(x,x') =\frac{w(x)w(x')}{x'-\lambda-iy} \delta(x-x').
\]
The limit of this operator at $y\downarrow 0$ is related to the behavior of a Cauchy-type integral. According to the Privalov-Plemel theorem, for functions satisfying the H\"older condition, the limit exists and is expressed in terms of a singular integral (the main value according to Cauchy):
\[
\lim_{y\downarrow 0} \int_{I_{\mathrm{near}}} \frac{\phi(x)}{x-\lambda-iy} dx = \pv \int_{I_{\mathrm{near}}} \frac{\phi(x)}{x-\lambda} dx + i\pi \phi(\lambda).
\]
The H\"older conditions on $w$ and the density of $\rho$ ensure that the principal value operator is bounded. This allows us to prove convergence in the operator norm for contributions from $I_{\mathrm{near}}$. Combining both contributions, we obtain the existence of a limit for $T_{\lambda+iy}$.
\end{proof}

\bigskip

{\bf 4.2. Case II: $\lambda\in\specp(H)$ (eigenvalue).}

Now let $\lambda$ be an isolated eigenvalue of $H$ of finite multiplicity $n = \dim V_\lambda < \infty$, where $V_\lambda = \ker(H-\lambda I)$. Let $P_\lambda$ be an orthoprojector on $V_\lambda$.

In this case, the limit of the operator $T_z$ normally \textbf{does not exist}.

\begin{proposition}
If $\lambda\in \specp(H)$ and $F P_\lambda\neq 0$, then $\norm{T_{\lambda+iy}}\to\infty$ for $y\downarrow 0$.
\end{proposition}

\begin{proof}
Decompose the space $\HH = V_\lambda \oplus V_\lambda^\perp$. The resolvent $(H-z)^{-1}$ splits into two parts. On $V_\lambda$, the $H$ operator acts as a multiplication by $\lambda$, so
\[
P_\lambda (H-z)^{-1} P_\lambda = P_\lambda (\lambda - z)^{-1} P_\lambda = \frac{1}{\lambda - (\lambda+iy)} P_\lambda = \frac{i}{y} P_\lambda.
\]
Consider the contribution of this term to $T_z$:
\[
F P_\lambda (H-z)^{-1} P_\lambda F^* =\frac{i}{y} (F P_\lambda) (F P_\lambda)^*.
\]
The operator $(F P_\lambda) (F P_\lambda)^*$ is a non—negative operator of rank no higher than $n$. If $F P_\lambda\neq 0$, its norm is strictly positive. Then
\[
\norm{T_{\lambda+iy}} \ge \norm{\frac{i}{y} (F P_\lambda) (F P_\lambda)^*} = \frac{1}{y} \norm{F P_\lambda}^2.
\]
Since $\norm{F P_\lambda} > 0$, the norm diverges at $y\downarrow 0$.
\end{proof}

Nevertheless, it is possible to prove the existence of a limit for the "regularized" part of the operator. Let $H_{\perp} = H\mid_{V_\lambda^\perp}$ and $F_{\perp} = F \mid_{V_\lambda^\perp}$.

\begin{theorem}
Let $\lambda\in\specp(H)$ be an isolated eigenvalue of finite multiplicity. Let $P_\lambda$ be a projector onto the corresponding proper subspace. Then the limit
is \[
\lim_{y\downarrow0} F(1-P_\lambda)(H-\lambda-iy)^{-1}(1-P_\lambda)F^*
\]
exists in the operator norm if the spectral measure $H$ in the vicinity of $\lambda$ satisfies the H\"older condition on $V_lambda^\perp$.
\end{theorem}

\begin{proof}
Operator $T_z^{\perp} := F(1-P_\lambda)(H-z)^{-1}(1-P_\lambda)F^*$ can be considered as a sandwiched resolvent for the operator $H_{\perp}$ in the space $\HH_{\perp} = V_\lambda^\perp$. By construction, $\lambda$ is not an eigenvalue for $H_{\perp}$. Thus, we have reduced the problem to Case I. If there is a continuous spectrum of $H$ (which coincides with the spectrum of $H_{\perp}$) satisfies the H\"older condition at the point $\lambda$, then by Theorem 4.2 the limit is $T_{\lambda+i0}^{\perp}$ exists.
\end{proof}

\begin{remark}
The divergence associated with the eigenvalue has a clear physical meaning. It corresponds to the resonant response of the system at a frequency that matches its natural frequency. Applications often do not study the limit itself, but the asymptotic of the behavior of $T_{\lambda+iy}$ or its imaginary part (which is related to the spectral density according to the Stone formula).
\end{remark}

\end{document}